\documentclass{article}
\usepackage[utf8]{inputenc}
\usepackage[margin=2.5cm]{geometry}
\usepackage{amsmath,amsthm,amsfonts,amssymb}
\usepackage[mathscr]{euscript}
\usepackage{amsmath}
\usepackage[version=4]{mhchem}
\usepackage{array,float}

\usepackage{graphicx,latexsym}
\usepackage{lscape}
\usepackage{multicol}
\usepackage{subfig}
\usepackage{graphicx}
\usepackage{multirow}
\usepackage{graphicx}
\usepackage{cite}

\newtheorem{thm}{Theorem}

\newtheorem{lem}[thm]{Lemma}

\newtheorem{discu}{Discussion:}

\newcommand{\diam}{{\rm diam}}

\begin{document}
\title{\textbf{Power Domination and Resolving Power Domination of Fractal Cubic Network}}
\author{\begin{tabular}{rcl}
		\textbf{S. Prabhu$^{\text a, }$\thanks{Corresponding author: drsavariprabhu@gmail.com }, A.K. Arulmozhi$^{\text b}$, Michael A. Henning$^{\text c}$, M. Arulperumjothi$^{\text d}$}
	\end{tabular}\\
	\begin{tabular}{c}
		$^{\text a}$\small Department of Mathematics, Rajalakshmi Engineering College, Chennai 602105, India \\
		$^{\text b}$\small Department of Mathematics, RMK College of Engineering and Technology, Puduvoyal 601206, India \\
		$^{\text c}$\small Department of Mathematics and Applied Mathematics, University of Johannesburg, Auckland Park, 2006 South Africa \\ 
		$^{\text d}$\small Department of Mathematics, St. Joseph's College of Engineering, Chennai 600119, India \\
\end{tabular}}
\maketitle
\vspace{-0.5 cm}
\begin{abstract}
{\small
In network theory, the domination parameter is vital in investigating several structural features of the networks, including connectedness, their tendency to form clusters, compactness, and symmetry. In this context, various domination parameters have been created using several properties to determine where machines should be placed to ensure that all the places are monitored. To ensure efficient and effective operation, a piece of equipment must monitor their network (power networks) to answer whenever there is a change in the demand and availability conditions. Consequently, phasor measurement units (\textbf{PMU}s)  are utilised by numerous electrical companies to monitor their networks perpetually. Overseeing an electrical system which consists of minimum \textbf{PMU}s is the same as the vertex covering the problem of graph theory, in which a subset $D$ of a vertex set $V$ is a power dominating set (\textbf{PDS}) if it monitors generators, cables, and all other components, in the electrical system using a few guidelines. Hypercube is one of the versatile, most popular, adaptable, and convertible interconnection networks. Its appealing qualities led to the development of other hypercube variants. A fractal cubic network is a new variant of the hypercube that can be used as a best substitute in case faults occur in the hypercube, which was wrongly defined in [Eng. Sci. Technol. \textbf{18}(1) (2015) 32-41]. Arulperumjothi et al. have recently corrected this definition and redefined this variant with the exact definition in [Appl. Math. Comput. \textbf{452} (2023) 128037]. This article determines the \textbf{PDS} of the fractal cubic network. Further, we investigate the resolving power dominating set (\textbf{RPDS}), which contrasts starkly with hypercubes, where resolving power domination is inherently challenging.}
\end{abstract}

\section{Introduction}
The electrical nodes and connection wires constitute an electrical power network. Electric power corporation must continuously monitor their systems' conditions. It is necessary to regulate the deployment of \textbf{PMU}s at precise position within the device.

Owing to the rising cost of \textbf{PMU}s, it is essential to engage as few as feasible while still tracking the entire system. This problem is introduced as a theoretical problem in \cite{HaHeHe02}, and coined it as a power domination problem after \cite{BaMiBo93} evince its existence.

Let $G$ be a simple connected graph whose vertex set (electrical hubs) and edge set (connecting cables), respectively, is denoted by $V(G)$ and $E(G)$. If the graph $G$ is clear from the context, we write $V = V(G)$ and $E = E(G)$. Two vertices $u$ and $v$ of $G$ are \emph{adjacent} if $uv\in E$. Two adjacent vertices are called \emph{neighbors}. For $r \ge 1$ and for a vertex $v\in V$, the \emph{open} $r$-\emph{neighborhood} of $v$ is denoted by $N_r(v)$ and is defined as the collection of vertices that are at distance $r$ from $v$. The \emph{closed} $r$-\emph{neighborhood} of $v$ is $N_r(v)\cup \{ v \}$, which is formally denoted by $N_r[v]$. When $r = 1$, the \emph{open} $r$-\emph{neighborhood} of $v$ is called its \emph{open neighborhood} and the \emph{closed} $r$-\emph{neighborhood} of $v$  is called its \emph{closed neighborhood}. For a set $S \subseteq V$, its \emph{open neighborhood} is the set $N_G(S) = \cup_{v \in S} N_1(v)$, and its \emph{closed neighborhood} is the set $N_G[S] = N_G(S) \cup S$.  The \emph{degree} of a vertex $v$ in $G$ is the number of vertices adjacent to $v$ in $G$, and is denoted by $\deg_G(v)$, and so $\deg_G(v) = |N_1(v)|$. We let $\mathbb{N}_n:=\{1,2,\ldots, n\}$ and $\mathbb{W}_n:=\{0,1,2,\ldots, n\}$. Also, we use $\mathbb{N}_{2n+1}-\mathbb{N}_{n}$ to denote $\{n+1,n+2,\ldots,2n+1\}$.

A set $S \subseteq V$ is a \emph{dominating set} of $G$ if every vertex in $V \setminus S$ is adjacent to at least one vertex in $S$. A vertex $v$ in $G$ \emph{dominates} itself and all its neighbors, and a set $X$ \emph{dominates} a set $Y$ in $G$ if every vertex in $Y$ is dominated by at least one vertex in $X$. Thus a dominating set $S$ of $G$ dominates every vertex in $V(G)$. The \emph{domination number} of $G$, denoted by $\gamma(G)$, is the minimum cardinality of a dominating set in $G$. We refer \cite{Ch18, ShLiYi14, ChLuWu14, PaZhZh10, Ye03, TsLiHs07, WuShLi10,haynes1,haynes2,haynes3} for a detailed study of domination and its variants.

If $V$ is observed or recursively observed by the following two rules with respect to a set $S$ of vertices in $G$, then $S$ is called \emph{power dominating set}, abbreviated \textbf{PDS} of $G$.

\begin{enumerate}
	\item \textbf{Domination:}\\
	$M \leftarrow  N_G[S]$
	\item \textbf{Propagation:}\\
	$\exists, \enspace x \in M$ s.t. $N(x) \cap (V- M)=\{y\}$\\
	$M\leftarrow M\cup\{y\}$
\end{enumerate}

In the domination step of PDS, all vertices in $N_G[S]$ are \emph{monitored}. In the event of the propagation phase of PDS, whenever a vertex $x$ is monitored, and $N_G(x)\smallsetminus \{y\}$ is monitored, then the vertex $y$ is added to the set $M$ and is also monitored. An initial set $S$ in the first step is a \textbf{PDS} for $G$ if upon completion of the propagation phase the resulting set $M$ is the entire vertex set~$V$. The term $\gamma_P(G)$ is used to denote the minimum cardinality of a \textbf{PDS} in $G$ called the \emph{power domination number}. 

Investigation of \textbf{PDS} for arbitrary general graphs is NP-complete. It remains NP-complete for classes of graphs such as bipartite, chordal, and split graphs~\cite{HaHeHe02}. Numerous algorithms for fetching the \textbf{PDS} for a particular graph family were reported in \cite{HaHeHe02, GuNiRa08, AaSt09}. This invariant is investigated for generalized Petersen family of graphs \cite{BaFe11,XuKa11,ChDoMo12,ZhShKa20}, hypercubes \cite{DeIlRa11}, circular-arc graphs \cite{LiLe13}, block graphs \cite {XuKaSh06}, permutaion graphs \cite{Wi19}, grids \cite{DoHe06}, planar and maximal planar graphs \cite{DoGoPe19}, Hanoi and Knödel graphs \cite{VaViHi18}, Kautz and de Bruijn graphs \cite{KuWu15}, claw-free regular graphs \cite{LuMaWa20} and octahedral structures \cite{PrDeEl22}. Also, for graph operations like strong and tensor product  \cite{DoMoKl08}, Cartesian product \cite{BaFe11, KoSo19}, join and corona product \cite{YuAgWa19}, the \textbf{PDS} problem is examined. The lower bounds \cite{FeHoKe17}, upper bounds \cite{ZhKaCh06}, and \textbf{NG} (Nordhaus-Gaddum) type results on \textbf{PDS} were discussed in \cite{BeFeFl18}. An excellent survey on this topic can be seen in the book chapter by Dorbec~\cite{Do-20}.

The concept of a $k$-\emph{power dominating set}, abbreviated $k$-\textbf{PDS}, was introduced in \cite{ChDoMo12}, where a $0$-\textbf{PDS} is a traditional dominating set and a $1$-\textbf{PDS} is the original power dominating set \textbf{PDS}. This parameter is investigated for certain interconnection networks \cite{RaArPr23}, weighted trees \cite{ChLuZh20},  Sierpi\'{n}ski networks \cite{DoKl14}, block graphs \cite{WaChLu16}, regular graphs \cite{DoHeMo13}. Few more variation namely  power dominating throttling \cite{BrCaHi19} and infectious power domination \cite{Bj20} are recent and interesting problems.

The concept of metric dimension (\textbf{MD}) was primarily discussed in \cite{Sl75} and separately in \cite{HaMe76}. It is equivalent to the least number of landmark vertices from which any two vertices can be differentiated using the distance parameter. The applications of this problem arise in various branches of science and technology.

The \emph{graph geodesic} between two vertices $u$ and $v$ is the length (in terms of the number of edges) of the shortest path between $u$ and $v$. The \emph{diameter} $\diam(G)$ of $G$ is the maximum length of a geodesic in $G$. The maximum distance over all pairs of $V(G)$ is called diameter.

For a vertex $x\in V(G)$, the \textit{code} of $x$ with respect to the subset $R$ $=\{r_{1},r_{2},\ldots,r_{k}\}$ of $V(G)$ is termed as a $k$-vector%
\begin{equation*}
	C_{R}(x)=(d_{G}(x,r_{1}),d_{G}(x,r_{2}),\ldots,d_{G}(x,r_{k}))
\end{equation*}%
where $d_{G}(x,r_{i})$ is the geodesic from $x$ to $r_{i}$ for $i \in \mathbb{N}_k$. The proper subset $R$ is a \textit{basis} or \textit{resolving set} for $G$ if any two distinct vertices of $G$ have nonidentical codes w.r.t $R$. Equivalently, for any two distinct vertices $x,y \in V(G)$, there exist a vertex $r\in R$ such that $d_{G}(r,x)\ne d_{G}(r,y)$. See Figure \ref{basis}. The basis of $G$ with fewer cardinality is called the \emph{resolving number} (also called the \emph{basis dimension}) of $G$ and is denoted by $\dim(G)$. More on this topic and some of its recent works can be found in \cite{ArKlPr23, PrJeAr23, PrMaAr22, PrDeAr22}.

\begin{figure}[H]
	\centering
	\includegraphics[scale=0.55]{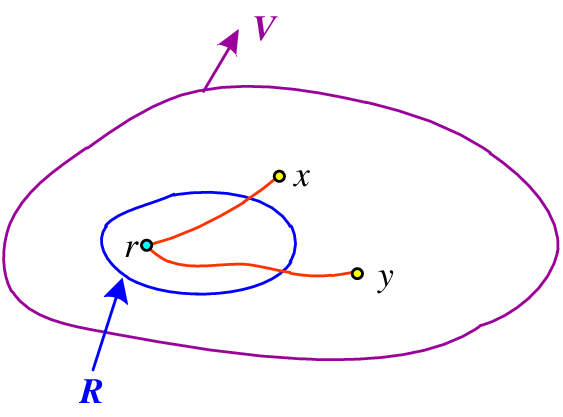}
	\caption{Basis}
	\label{basis}
\end{figure}

\section{Fractal Cubic Network: A New Hypercube Variant}

\par Frequently, an interconnection networks with multiprocessors are essential to link an eloquent portion of dependably imitated processors (vertices). In lieu of shared memory, message transient is mostly used to afford complete transmission and synchronisation between processors for programmed execution. Owing to the availability of cost-effective, potent memory circuits and microprocessors, there has been a recent surge in fascination in implementing and designing multistage networks.

In parallel computing and supercomputers, interconnection networks, the development of CPUs, and routing algorithms are three primary research areas. The interconnection network, which connects millions of processors, is essential for the development of a supercomputer.

Multiple processors, each with their own cognitive connections (edges) and local sensing that enable data transfer between processors (vertices), compose an interconnection network. It can be represented as the previously defined graph $G$ in which two vertices $x_i$ and $x_j$ are directly connected by a communication connection. The attribute used to evaluate the productiveness of the networks are the bisection width, broadcasting time, fault-tolerance,  degree, and diameter \cite{AkKr89}.

The hypercube is a common interconnection network design distinguished by its regularity, ease of transit, recursive structure, symmetry, and high connectedness. In recent years, hypercubes have been the subject of extensive research into their various properties \cite{HaHaWu88}.

There are many hypercube variations in the literature, including exchanged hypercube \cite{LoHsPa05}, folded hypercubes \cite{ElLa91, ZhLiXu08}, crossed cubes \cite{Ef92, FaJi07}, exchanged crossed cube \cite{LiMuLi13}, twisted cubes \cite{AbPa91, ChWaHs99}, locally twisted cubes \cite{YaEvMe05},  shuffle cubes \cite{LiTaHs01}, spined cubes \cite{ZhFaJi11}, Möbius cubes \cite{CuLa95}, and augmented cubes \cite{ChSu02}. The hierarchical cubic network (HCN) \cite{GhDe95, YuPa98} and its folded version in \cite{DuChFa95} have also been ideas put forth based on a hierarchical framework employing the base hypercube as an introductory module.

Despite the fact that numerous studies have been conducted on the variants of hypercube enumerated above, the problem resolving number has not been investigated for any of these variants except fractal cubic network. With this motivation, we investigate power domination and resolving power domination for the newly proposed hypercube variant fractal cubic network (\textbf{FCN}) \cite{KaSe15}. Though the definition of this architecture is not clear in~\cite{KaSe15}, we corrected this definition in~\cite{ArKlPr23}. We define $FCN(0)$ as a cycle with four vertices $00$, $01$, $11$, and $10$. For $d \ge 1$, we define $FCN(d)$ as follows:

An $d$-dimensional \textbf{FCN} is defined as $FCN(d) = (V_{1}(d), E_{1}(d))$, $d > 0$, and can be constructed as follows
\[
FCN(d) = 11 \mathbin\Vert FCN(d-1)\cup 01 \mathbin\Vert FCN(d-1)\cup 10 \mathbin\Vert FCN(d-1) \cup 00 \mathbin\Vert FCN(d-1),
\]
where
\[
V_{1}(d) = 11 \mathbin\Vert V_{1}(d-1)\cup 01 \mathbin\Vert V_{1}(d-1)\cup 10\mathbin\Vert V_{1}(d-1)\cup 00\mathbin\Vert V_{1}(d-1)
\]
and
\[
\begin{array}{lcl}
E_{1}(d) & = & 11 \mathbin\Vert E_{1}(d-1) \cup 01 \mathbin\Vert E_{1}(d-1)\cup 10\mathbin\Vert E_{1}(d-1)\cup 00\mathbin\Vert E_{1}(d-1) \\
& & \, \cup \, \{(001100\ldots0,101100\ldots0), (101100\ldots0,111100\ldots0) \} \\
& & \, \cup \, \{(111100\ldots0,011100\ldots0),(011100\ldots0,001100\ldots0)\}.
\end{array}
\]
Figure~\ref{twins}(a)-(d), respectively, denotes the $FCN(0)$, $FCN(1)$, $FCN(2)$, and $FCN(3)$.

We denote $11 \mathbin\Vert FCN(d-1)$ as the collection of strings obtained by concatenating $11$ and each of the strings in $V_1(d-1)$. For $d \ge 1$, $FCN(d)$ is constructed from four copies of $FCN(d-1)$ with four additional edges connecting them.
\section{Main Results}

Before proceeding to the main arguments, we first present some preliminary lemma and results that are crucial for our investigation.

\subsection{Twin Nodes} \label{methods}

Two vertices $u, v\in V$ are said to be \emph{non}-\emph{adjacent twins} (also called \emph{open twins} in the literature) if $N_1(u)=N_1(v)$ and are \emph{adjacent twins} (also called \emph{closed twins} in the literature) if $N_1[u]=N_1[v]$ (see~\cite{BaEsRa19, Li18}). The vertices $u$ and $v$ shown in Figure~\ref{twindefn}(a) are open twins, while the vertices  $u$ and $v$ shown in Figure~\ref{twindefn}(b) are closed twins. Two vertices in $G$ are \emph{twins} if they are open or closed twins in $G$. A set $T \subseteq V(G)$ is a \emph{open twin set} or \emph{open twin class} of $G$, if every pair of vertices in $T$ are open twins in $G$, while a set $T \subseteq V(G)$ is a \emph{closed twin set} or \emph{closed twin class} of $G$, if every pair of vertices in $T$ are closed twins in $G$. The sets $T_1$ and $T_2$ shown in Figure~\ref{twindefn}(c) are examples of two open twin sets.

Arulperumjothi, Klavžar, and Prabhu~\cite{ArKlPr23} determined $\dim(G)$, where $G = FCN(d)$ and $d \ge 1$.

\begin{thm}{\rm \cite{ArKlPr23}}\label{mmdfcn}
	For $d>0$,  $\dim(FCN(d))=4^{d}$.
\end{thm}

\begin{figure}[H]
	\centering
	\subfloat[]{\includegraphics[scale=0.55]{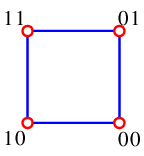}}
	\quad \quad  \quad
	\subfloat[]{\includegraphics[scale=0.55]{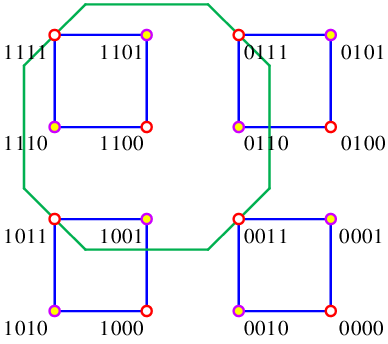}}
	\quad \quad \quad
	\subfloat[]{\includegraphics[scale=0.55]{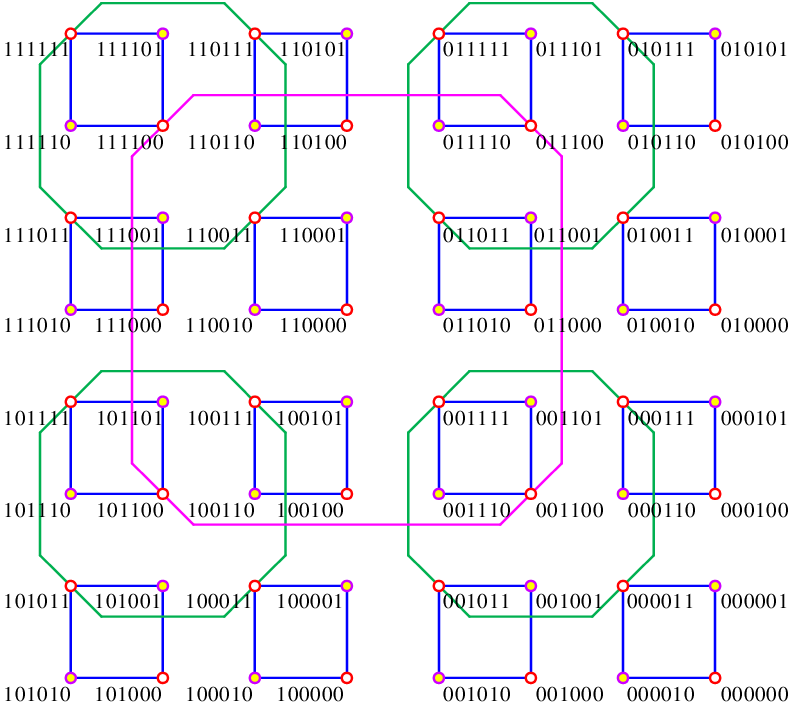}}
	\quad
	\subfloat[]{\includegraphics[scale=0.56]{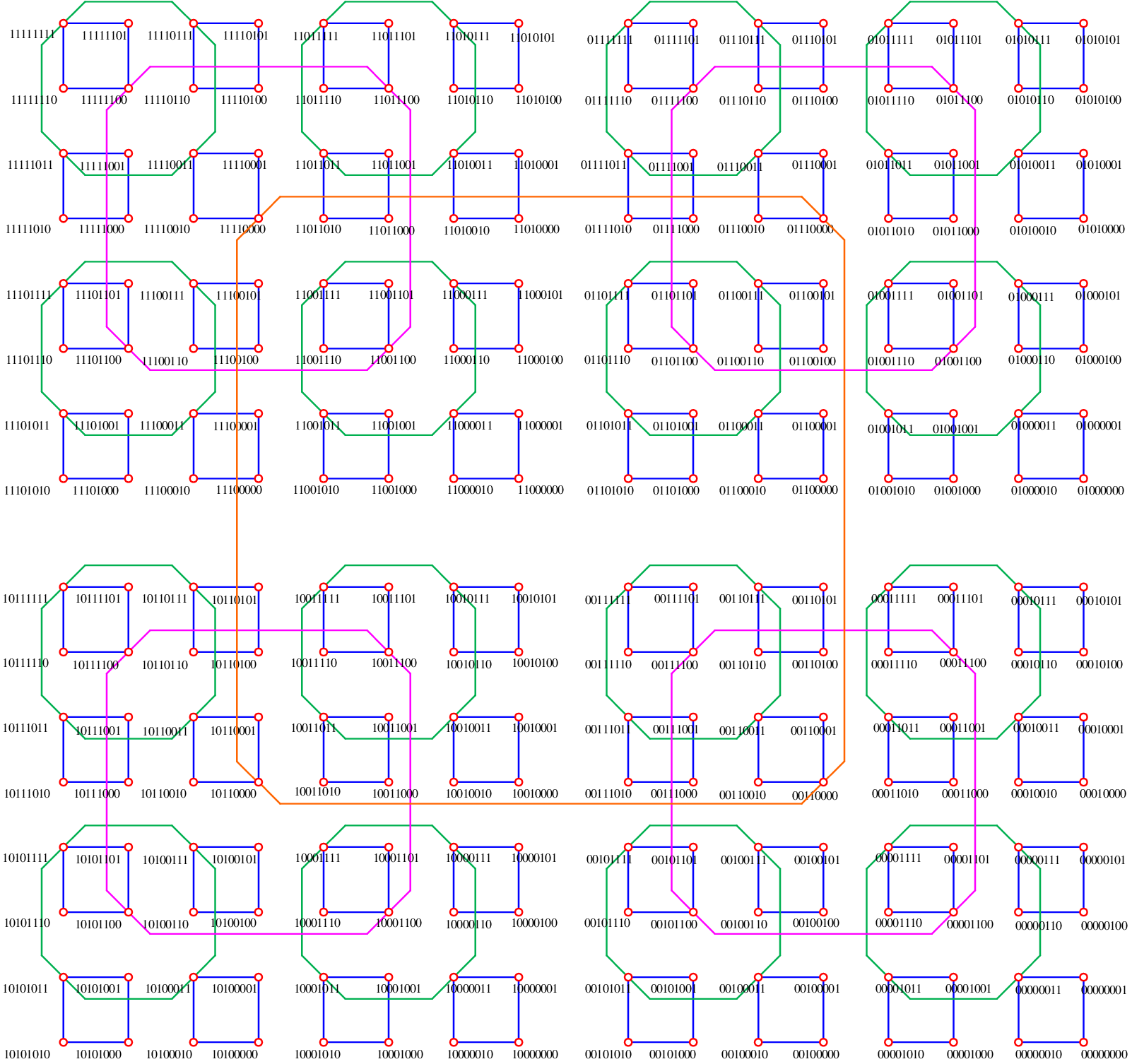}}
	\caption{(a) $FCN(0)$; (b) $FCN(1)$; (c) $FCN(2)$; (d) $FCN(3)$}
	\label{twins}
\end{figure}

\begin{figure}[H]
	\centering
	\subfloat[]{\includegraphics[scale=0.6]{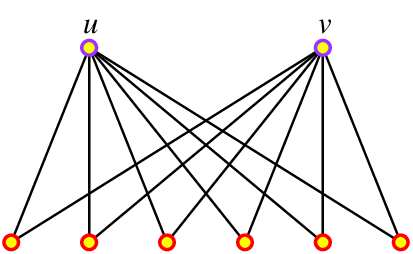}}
	\quad \quad \quad
	\subfloat[]{\includegraphics[scale=0.6]{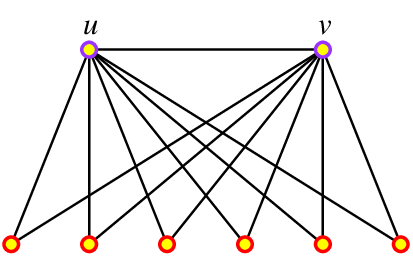}}
	\quad \quad \quad
	\subfloat[]{\includegraphics[scale=0.6]{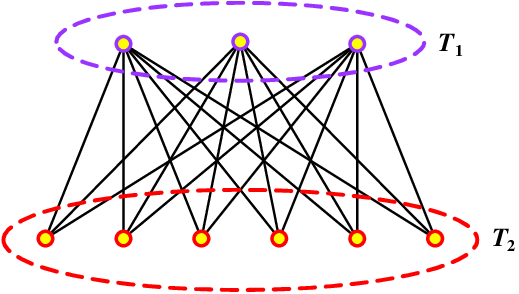}}
	\caption{(a) Open twins;  (b) Closed twins; (c) Open twin sets $T_1$ and $T_2$} \label{twindefn}
\end{figure}

We proceed further with the following two key lemmas.

\begin{lem}
\label{pds:non-emptylemma}
Let $G$ be a connected graph. If $S$ is a \textbf{PDS} of $G$ and $T$ is an open twin class of $G$, then $S\cap N_1(T) \ne \emptyset$ or $|S\cap T|\ge |T|-1$.
\end{lem}
\begin{proof}	
Let $G$ be a connected graph, and let $S$ be a \textbf{PDS} of $G$ and $T$ an open twin class of $G$. We note that $T$ is an independent set and $|T| \ge 2$. Moreover every $x \in N_1(T)$ is adjacent to every $y \in T$. Thus, every path connecting a vertex in $T$ and a vertex in $V(G) - T$ must contain a vertex in $N_1(T)$. We show that $S\cap N_1(T) \ne \emptyset$ or $|S\cap T|\ge |T|-1$. Suppose, to the contrary, that $S \cap N_1(T) = \emptyset$ and $|S \cap T| \le |T| - 2$. Let $R = T \setminus S$ (possibly, $R = T$ which occurs is $S \cap T = \emptyset$), and so $|R| = |T| - |S| \ge 2$. Since $T$ is an independent set and since $S \cap N_1(T) = \emptyset$, no vertex in $S$ is adjacent to a vertex in $R$. Thus, $R$ contains no vertex in $N_G[S]$, and so no vertex of $R$ is dominated in the initial domination step of \textbf{PDS}. As observed earlier, every path connecting a vertex in $T$ and a vertex in $V(G) - T$ must contain a vertex in $N_1(T)$. Hence in order to observe the vertices in $R$ we must first observe a vertex in $N_1(T)$ either by domination or by propagation from the vertices that do not belong to the set $R$. However, every vertex in $N_1(T)$ is adjacent to all vertices of $R$. Thus the vertices of $R$ cannot be propagated as every vertex of $N_1(T)$ has $|R| \ge 2$ vertices as neighbors. This contradicts the supposition that $S$ is a \textbf{PDS} of $G$.
\end{proof}

\begin{lem}
\label{pds:lbd}
If $G$ is a connected graph with $k$ vertex disjoint twin classes $T_1, \ldots, T_k$ such that $N_1(T_i) \cap N_1(T_j) = \emptyset$ for every $1 \le i < j \le k$, then $\gamma_P(G) \ge k$.
\end{lem}
\begin{proof}	
Let $S$ be a \textbf{PDS} of $G$. By Lemma~\ref{pds:non-emptylemma}, $S \cap N_1(T_i) \ne \emptyset$ or $|S \cap T_i| \ge |T_i|-1$ for all $i \in \mathbb{N}_k$. Thus since $|T_i| \ge 2$, we note that if $S \cap N_1(T_i) = \emptyset$, then $|S \cap T_i| \ge 1$. Hence at least one of $|S \cap N_1(T_i)| \ge 1$ or $|S \cap T_i| \ge 1$ holds for all $i \in \mathbb{N}_k$, implying that $|S \cap (T_i \cup N_1(T_i))| \ge 1$ for all $i \in \mathbb{N}_k$. By supposition, the twin classes $T_1, \ldots, T_k$ are vertex disjoint. Further, the sets $T_i \cup N_1(T_i)$ and $T_j \cup N_1(T_j)$ are vertex disjoint for every $1 \le i < j \le k$. From this we infer that
\[
|S| \ge \sum_{i = 1}^k |S \cap (T_i \cup N_1(T_i))| \ge k.
\]
Since $S$ in arbitrary \textbf{PDS} of $G$, we infer that $\gamma_P(G) \ge k$.
\end{proof}

We are now in a position to prove the following result.

\begin{thm}\label{pdfcn}
For $d>0$,  $\gamma_P(FCN(d))=4^{d}$.
\end{thm}
\begin{proof}

Let $G = FCN(d)$. From the definition of $G$, we note that $N_{1}(u_{2d+2}u_{2d+1}\ldots u_301)=N_{1}(u_{2d+2}u_{2d+1}\ldots u_310)$, where $u_{2d+2}u_{2d+1}\ldots u_3\in \{0,1\}^{2d}$. That is,
$\{u_{2d+2}u_{2d+1}\ldots u_301,u_{2d+2}u_{2d+1}\ldots u_310\}$ is an open twin set in $G$. The $FCN(2)$ and its $2^4$ twin sets are marked in Figure~\ref{twinsinFCN}.
Hence, $G$ contains $2^{2d}$ twin sets, each containing exactly two vertices.  Therefore by Lemma~\ref{pds:lbd}, $\gamma_P(G)\ge 4^{d}$.

\begin{figure}[ht!]
\centering
\includegraphics[scale=0.8]{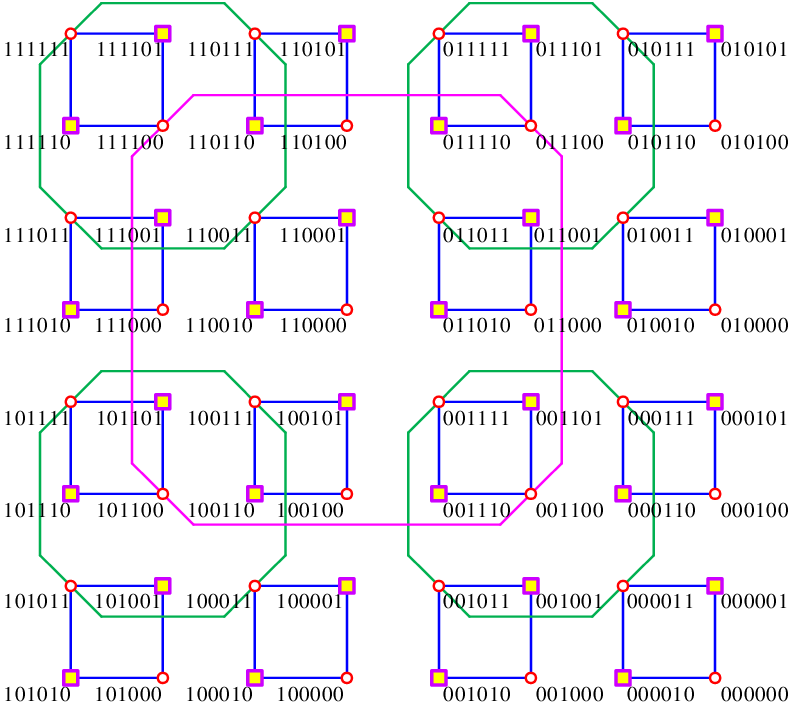} \quad \caption{The $2^4$ open twin sets in $FCN(2)$}\label{twinsinFCN}
\end{figure}
	
To show that $\gamma_P(G)\le 4^{d}$, let 
\[
A = \bigcup \{u_{2d+2}u_{2d+1}\ldots u_3 01 \} 
\hspace*{0.5cm} \mbox{and} \hspace*{0.5cm} 
B = \bigcup \{u_{2d+2}u_{2d+1}\ldots u_3 10 \} 
\]
where the union is taken over all $u_{2d+2}u_{2d+1}\ldots u_3\in \{0,1\}^{2d}$. We note that $|A| = |B| = 4^{d}$. Let $C = V(G) \setminus (A \cup B)$. We claim that the set $A$ is a power dominating set of $G$. For this, we have to prove that every vertex of $V$ is monitored by the vertices of $A$ either by domination or by propagation. We note that $N_1(u_{2d+2}u_{2d+1}\ldots u_301)=\{u_{2d+2}u_{2d+1}\ldots u_311, u_{2d+2}u_{2d+1}\ldots u_300\}$, where as before $u_{2d+2}u_{2d+1}\ldots u_3\in \{0,1\}^{2d}$. Hence, $N_G[A] = A \cup C$. Thus, all vertices in $A \cup C$ are monitored by the vertices of $A$ by domination. We note further that $A \cup B$ is an independent set and that $N_G(B) = C$. Also, each vertex $x \in \{u_{2d+2}u_{2d+1}\ldots u_3 10  \colon  u_i \in \{0,1\} \}$ is of degree~$2$ in $FCN(d)$ and is adjacent to vertex $y$, where $y\in N_1(u_{2d+2}u_{2d+1}\ldots u_301)$. Moreover, the vertices in $B$ are pairwise at distance at least~$3$ apart, and so no two vertices in $B$ have a common neighbor that belongs to $C$. Thus every dominated vertex in $C$ has all its neighbors dominated, except for exactly one vertex in $B$. Thus, each dominated vertex in $C$ propagates to its unique neighbor in $B$. Thus, each vertex $y$ is propagated by either  $u_{2d+2}u_{2d+1}\ldots u_311$ or $u_{2d+2}u_{2d+1}\ldots u_300$. Hence, $V = M(A)$, that is, $A$ is a power dominating set of $G$, implying that $\gamma_P(G) \le |A| = 4^{d}$. As observed earlier, $\gamma_P(G) \ge 4^{d}$. Consequently, $\gamma_P(G) = 4^{d}$.
\end{proof}

\section{Resolving Power Domination of FCN}

This section recalls the definition of a resolving-power dominating set and its minimum cardinality. A subset $S$ is both resolving and power-dominating. The subset $S$ of $V$ is called a resolving-power dominating set, and its minimum cardinality is the resolving power domination number and is notated by $\eta_ {P}(G)$. The first paper on this notion was introduced in \cite{StRaCy15}. 

\begin{thm}\label{rps} \rm\cite{StRaCy15}
For any simple connected graph $G$, $\max \{\dim(G), \gamma_P(G)\}\le \eta_{P}(G) \le \dim(G)+\gamma_P(G)$.
\end{thm}

Since then, we are not aware of further research on this problem. Prabhu et al. recently investigated this parameter for probabilistic neural networks (PNN) in \cite{PrDeAr22}. As a consequence of our main results in this paper, we determine the resolving power domination number of a fractal cubic network.

\begin{thm}
For $d \ge 1$, $\eta_{P}(FCN(d))=4^d$.	
\end{thm}
\begin{proof}
The proof follows immediately from Theorems~\ref{mmdfcn},~\ref{pdfcn} and~\ref{rps}.
\end{proof}

\section{Conclusion}

\par Multistage interconnection networks are essential to parallel computing because their performance on a large scale is determined by their connectivity. Communication efficacy is a key prosecution indicator in parallel computing. The diameter of a network's interconnections is an essential metric of transmission efficiency. Hypercube is prevalent in all architecture owing to its advantageous properties. There are numerous prospective variants of this network that can be created by modifying certain links. This variant of hypercube is one such variant. In the present investigation, we concentrate on power domination and resolving-power domination for this new invariant in the most optimal manner. In this paper, we completely determine these parameters for fractal cubic networks. 

\textbf{Compliance with ethical standards} \\
\textbf{Conflict of interest:} The authors declare that there is no conflict of interest regarding the publication of this paper.\\
\textbf{Data Availability Statement:} No Data associated with the manuscript.

\end{document}